\documentclass[12pt]{article}
\usepackage{stmaryrd}
\usepackage{amsmath}
\usepackage{graphicx}
\usepackage{color}

\usepackage{authblk}
\usepackage[title]{appendix}
\usepackage{amsfonts,amscd,amssymb, mathtools,mathrsfs, dsfont}

\usepackage{pgfplots}
\usepackage{float}

\usepackage{url}
\usepackage{hyperref}
\usepackage{cite}
\usepackage{array}

\hypersetup{
colorlinks=true, 
breaklinks=true, 
urlcolor= blue, 
linkcolor= red, 
citecolor=purple, 
pdftitle={}, 
pdfauthor={}, 
pdfsubject={} 
}

\tikzset{elegant/.style={smooth,thick,samples=50,cyan}}
\tikzset{eaxis/.style={->,>=stealth}}

\newif\ifblog
\newif\iftex
\blogfalse
\textrue

\newcommand{\thmref}[1]{Theorem~{\rm \ref{#1}}}

\newcommand{\figref}[1]{Figure~{\rm \ref{#1}}}

\def\wU{\widehat{U}}
\def\wwU{\widetilde{U}}

\def\d{{\rm d}}

\def\R{{\mathbb R}}
\def\RR{{\cal R}}

\def\Q{{\cal Q}}

\def\SS{{\cal S}}
\def\NS{{\cal NS}}
\def\HC{{\cal HC}}
\def\MC{{\cal MC}}
\def\LC{{\cal LC}}

\def\p{{\partial}}

\newcommand{\ep}{\varepsilon}
\newcommand{\al}{\alpha}

\newcommand{\la}{\lambda}

\newcommand{\si}{\sigma}
\newcommand{\ka}{\kappa}

\newcommand{\argmin}{\mathop{\rm argmin}\limits}

\def\hh{{\overline{h}}}
\def\uh{{u^{\overline{h}}}}

\def\uu{{\overline{u}}}
\def\uv{{\underline{v}}}

\def\fh{{\varphi^{\overline{h}}}}

\newtheorem{theorem}{Theorem}[section]
\newtheorem{lemma}[theorem]{Lemma}

\newtheorem{proposition}[theorem]{Proposition}

\newenvironment{proof}{\noindent {\sc Proof:}}{\strut\hfill $\Box$} 

\newcommand{\LL}{{\cal L}}
\newcommand{\TT}{{\cal T}}
\newcommand{\JJ}{{\cal J}}
\newcommand{\wt}{\widetilde}

\title{Optimal Consumption-Investment for General Utility with a Drawdown Constraint over a Finite-Time Horizon
\thanks{The work is supported by Guangdong Basic and Applied Basic Research Foundation (No. 2025A1515012146, and No. 2024A1515012430),
NNSF of China (No. 11901244).}}

\author[1]{Chonghu Guan}
\author[2]{Xinfeng Gu\thanks{Corresponding author. Email: martineuclid@gmail.com.}}
\author[2]{Wenhao Zhang}
\author[2]{Xun Li}

\affil[1]{School of Mathematics, Jiaying University, Meizhou 514015, China}
\affil[2]{Department of Applied Mathematics, The Hong Kong Polytechnic University, Hong Kong, China}
\date{}

\begin{document}
\maketitle

\begin{abstract}
We study an optimal investment and consumption problem over a finite-time horizon, in which an individual invests in a risk-free asset and a risky asset, and evaluate utility using a general utility function that exhibits loss aversion with respect to the historical maximum of consumption. Motivated by behavioral finance and habit formation theory, we model the agent's preference for maintaining a standard of living by imposing constraints on declines from the peak consumption level.

To solve the resulting Hamilton-Jacobi-Bellman (HJB) variational inequality, which is fully nonlinear, we apply a dual transformation, transforming the original problem into a linear singular control problem with a constraint. By differentiating the value function further, we reduce the constrained linear singular control problem to a linear obstacle problem. We prove the existence of a solution to the obstacle problem under standard constraints. It allows us to characterize the optimal consumption and investment strategies through piecewise analytical feedback forms derived from the dual formulation. 

Our analysis contributes to the literature on habit formation, drawdown constraints, and stochastic control by explicitly characterizing the time-dependent free boundaries and the associated optimal feedback strategies.
\bigskip

\noindent {\bf Mathematics Subject Classification.} 35R35; 91B70; 91B60.
\end{abstract}


\setlength{\baselineskip}{0.25in}

\section{Introduction}

Merton first studied optimal portfolio and consumption problems through utility maximization in \cite{Mer69}, and it has since become a classical topic in mathematical finance.
Subsequent research has expanded upon these foundations by considering various financial market models, utility functions, trading constraints, and the presence or absence of transaction costs. Initially, researchers focused on utility maximization with respect to consumption without considering consumption habit formation. However, since the 1970s, the study of consumption habit formation has become a central topic in mathematical finance. Detemple and Zapatero~\cite{DZ91} established optimal consumption-portfolio policies under habit-forming preferences, providing structural insights into optimal strategies for various utility functions. Detemple and Karatzas~\cite{DK03} extended this framework by introducing a non-addictive habit formation model, allowing consumption to fall below the habitual level, and analyzed the resulting optimal policies. Yu~\cite{Yu15} investigated utility maximization with addictive consumption habit formation in incomplete semimartingale markets, employing convex duality methods to prove the existence and uniqueness of optimal solutions. Yang and Yu~\cite{YY22} addressed a composite problem involving optimal entry timing and subsequent consumption decisions under habit formation, demonstrating that the value function satisfies a variational inequality and is the unique viscosity solution. These studies collectively advance our understanding of how habit formation influences consumption and investment decisions in various financial contexts.

Dybvig termed the case where the consumption rate is not allowed to fall below the habit as ``addictive" \cite{Dyb95}; otherwise, if the consumption rate is permitted to fall below the habit, it is referred to as ``non-addictive". Recent advancements in consumption and dividend optimization under drawdown constraints have been made by Arun~\cite{Arun12}, who extended the classical Merton problem by imposing a drawdown constraint on consumption, ensuring it does not fall below a fixed proportion of its historical maximum; Bahman, Bayraktar, and Young~\cite{BBY19}, who investigated optimal dividend distribution strategies under drawdown and ratcheting constraints, formulating the problem as a stochastic control task aimed at maximizing expected discounted utility until bankruptcy; Chen, Li, Yi, and Yu~\cite{CLYY23}, who addressed a finite-horizon utility maximization problem with a drawdown constraint on consumption, employing partial differential equation techniques to establish the existence and uniqueness of solutions; and Deng, Li, Pham, and Yu~\cite{DLPY22}, who explored an infinite-horizon optimal consumption problem with a path-dependent reference point, utilizing dual transformation methods to derive closed-form solutions for optimal strategies. Utility maximization problems with habit formation are more complex than their non-habit formation counterparts. 
The motivation behind considering habit formation is that an individual's satisfaction from consumption depends on previous consumption levels; a decline in consumption may lead to disappointment and discomfort. Various types of consumption references have been studied. One approach considers the weighted average of past consumption integrals, as found in Detemple and Karatzas~\cite{DK03} and Schroder and Skiadas~\cite{SS02}. Another approach focuses on the impact of the historical maximum consumption level. For instance, Arun~\cite{Arun12} extended the classical Merton problem by imposing a drawdown constraint on consumption, ensuring it does not fall below a fixed proportion of its historical maximum. Bahman, Bayraktar, and Young~\cite{BBY19} investigated optimal dividend distribution strategies under drawdown and ratcheting constraints, formulating the problem as a stochastic control task to maximize the expected discounted utility until bankruptcy. Chen, Li, Yi, and Yu~\cite{CLYY23} addressed a finite-horizon utility maximization problem with a drawdown constraint on consumption, employing partial differential equation techniques to establish the existence and uniqueness of solutions. Deng, Li, Pham, and Yu~\cite{DLPY22} explored an infinite-horizon optimal consumption problem with a path-dependent reference point, utilizing dual transformation methods to derive closed-form solutions for optimal strategies. Jeon and Park~\cite{JP21} generalized the work of Arun by considering general utility functions over an infinite horizon, and Jeon and Oh~\cite{JO22} extended this approach to a finite horizon model using a dual optimal stopping problem.

From a psychological and behavioral perspective, individuals often form consumption habits based on past peak consumption levels, and deviations below these peaks may cause discomfort or dissatisfaction. It motivates imposing a drawdown constraint, ensuring that consumption does not fall below a fixed proportion of its historical maximum. Rather than modeling the utility as a function of consumption relative to a reference point, our approach reflects the tendency of individuals to maintain a stable standard of living by enforcing minimum consumption standards through constraints.

Utility maximization problems have incorporated past consumption habits into the utility function. Deng, Li, Pham, and Yu~\cite{DLPY22} introduced an exponential utility to measure the distance between consumption control and a proportion of the past consumption maximum over an infinite horizon, aiming to depict the quantitative influence of relative changes on investor preferences.  Loss aversion and utility with a reference have recently garnered more attention in behavioral finance. Bilsen, Laeven, and Nijman~\cite{BLN20} examined a similar problem with a two-part utility function, treating conventional consumption habit formation as the consumption reference. Curatola~\cite{Cur17} studied a utility maximization problem under an S-shaped utility on consumption for a loss-averse agent with a reference generated by a specific integral of the past consumption process. Li, Yu, and Zhang~\cite{LYZ21} combined loss aversion on relative consumption with reference to historical consumption maxima using an S-utility over an infinite horizon to model different consumer feelings when consumption exceeds or falls below past spending maxima.

In contrast to the above works, which often assume CRRA-type utility functions to enable dimensionality reduction, we consider a general utility function \( U(c) \), assumed only to be strictly increasing, concave, and twice continuously differentiable, with marginal utility vanishing as consumption increases.
This generalization brings some technical and conceptual challenges to us. 
First, the standard techniques that reduce the HJB equation to a two-dimensional problem are no longer applicable. Instead, the problem becomes inherently three-dimensional in wealth \(x\), time \(t\), and the past consumption maximum \(h\).
This increase in dimensionality leads the free boundaries to become significantly more complex. The switching surfaces \( x^L(t,h) \), \( x^H(t,h) \), and \( x^*(t,h) \) are now full two-dimensional surfaces depending on both \(t\) and \(h\). We rigorously establish their continuity  on $[0,T)\times\R^+$ and monotonicity with respect to \(h\), and examine their behavior as \(t \to T\).
Second, establishing the strict concavity of the dual problem's solution and the asymptotic properties of its derivative presents a significant theoretical challenge in this study, requiring original analysis beyond previous results.

The remainder of this paper is organized as follows. In Section 2, we introduce the market model and formulate the optimal consumption problem, which involves general utility maximization with reference to the past spending maximum over a finite horizon. We also present the associated Hamilton-Jacobi-Bellman (HJB) variational inequality. Section 3 outlines our main results, detailing the solution's properties and the equation's free boundaries, along with their economic interpretations. In Section 4, we simplify the fully nonlinear HJB equation. Through dual transformation and variable substitution, we convert it into a linear equation with constant coefficients and a gradient constraint, resulting in a linear singular control problem. By differentiating the equation, we transform the gradient constraint variational inequality into a function constraint variational inequality, specifically a linear obstacle problem. Section 5 establishes the existence and uniqueness of the strong solution to the obstacle problem via the penalty approximation method and discusses the properties of the free boundary arising from the obstacle problem. In Section 6, we construct the classical solution to the singular control problem using the solution and the free boundary of the obstacle problem. Finally, in Section 7, we prove the verification theorem for the classical solution to the original HJB variational inequality and all associated thresholds of the wealth variable in analytical form with boundary conditions, utilizing the solution of the HJB equation.


\section{Model formulation}\label{sec:model}
\setcounter{equation}{0}

Let $(\Omega, \mathcal{F}, \mathbb{F}, \mathbb{P})$ be a filtered probability space, where the filtration $\mathbb{F} = (\mathcal{F}_t)_{t \in [0, T]}$ satisfies the usual conditions. The financial market consists of a risk-free asset $B_t$ and a risky asset $S_t$, whose prices evolve according to the stochastic differential equations
\[
\begin{aligned}
\mathrm{d}B_t &= r B_t\, \mathrm{d}t, \\
\mathrm{d}S_t &= \mu S_t\, \mathrm{d}t + \sigma S_t\, \mathrm{d}W_t,
\end{aligned}
\]
where $r \geq 0$ denotes the constant interest rate, and $\{W_t\}_{t \in [0, T]}$ is an $\mathbb{F}$-adapted Brownian motion. The constants $\mu > r$ and $\sigma > 0$ represent the drift and volatility of the risky asset, respectively.

Let $\Pi^t = \{\Pi_s\}_{s \in [t,T]}$ denote the investor’s allocation to the risky asset, and let $C^t = \{C_s\}_{s \in [t,T]}$ denote the consumption rate. Given an initial wealth $x > 0$, the wealth process $X_s$ evolves according to the self-financing equation:

\begin{equation}\label{Xt}
\left\{
\begin{array}{ll}
\d X_s=[r (X_s-\Pi_s)+\mu \Pi_s -C_s]\d s+\si \Pi_s \d W_s,\quad t\leq s\leq T, \\
\;X_t\;=x.
\end{array}
\right.
\end{equation}

Denote
$$H_s = \max\Big\{h, \max\limits_{t \leq \theta \leq s} C_\theta \Big\}$$
as the historical maximum of consumption, where $h > 0$ is the initial reference level.
Given a fixed proportion $0 \leq b \leq 1$, we say that the consumption strategy $\{C_s\}_{s \in [t,T]}$ satisfies a {\it drawdown constraint} if
\begin{equation}\label{HC}
 b H_s \leq C_s,\quad t \leq s \leq T.
\end{equation}

An admissible strategy $(\Pi^t, C^t)$ must satisfy the following conditions:
\begin{enumerate}
\item $(\Pi^t, C^t)$ is $\mathbb{F}$-progressively measurable, and the integrability condition $\mathbb{E}[\int_t^T (C_s + \Pi_s^2)\, \mathrm{d}s] < \infty$ holds;
\item $X_s \geq 0$ for all $s \in [t,T]$;
\item $C^t$ satisfies the drawdown constraint \eqref{HC}.
\end{enumerate}

The objective of the individual is to find an admissible strategy $(\Pi^t, C^t)$ that maximizes the utility of consumption and terminal wealth over the finite horizon $[t,T]$. The value function is defined as
\begin{equation}\label{value}
V(x,t,h) = \sup\limits_{\Pi^t, C^t} \mathbb{E} \left[\int_t^T e^{-\rho (s-t)} U(C_s)\, \mathrm{d}s + e^{-\rho (T-t)} U_T(X_T) \;\middle|\; X_t = x,\; H_t = h \right],
\end{equation}
where $\rho > 0$ is a constant discount rate. The terminal utility function $U_T(\cdot)$ is allowed to be identically zero as a special case.
For notational convenience, we will write $\mathbb{E}[\cdot \mid X_t = x,\; H_t = h]$ simply as $\mathbb{E}[\cdot]$ when there is no ambiguity.

In this paper, we assume that the utility function $U(\cdot)$ over consumption satisfies the following canonical conditions:
\begin{equation}\label{U1}
U(\cdot) \in C^2(\mathbb{R}^+),\quad U'(\cdot) > 0,\quad U''(\cdot) < 0,\quad U'(+\infty) = 0,
\end{equation}
\begin{equation}\label{U2}
U(c) \leq K\left(\frac{1}{1 - \gamma} c^{1 - \gamma} + 1\right) \quad \text{for some } 0 < \gamma < 1 \text{ and } K > 0.
\end{equation}

The above assumptions include several commonly used utility functions: the CRRA (Constant Relative Risk Aversion) utility $U(c) = \frac{1}{1 - \gamma} c^{1 - \gamma}$ with $\gamma > 0$, $\gamma \neq 1$; the logarithmic utility $U(c) = \ln(c)$; and the CARA (Constant Absolute Risk Aversion) utility $U(c) = -\frac{1}{\alpha} e^{-\alpha c}$ with $\alpha > 0$.

For the terminal utility function $U_T(\cdot)$, if $U_T(\cdot) \not\equiv 0$, we impose the same assumptions as in \eqref{U1}-\eqref{U2}.

Additionally, we require
\begin{equation}\label{U3}
U_T(x) \geq K\Big(\frac{1}{1-\theta} c^{1-\theta} - 1\Big) \quad \hbox{for\;some\;} \theta>1 \hbox{\;and\;} K>0, 
\end{equation}
to guarantee that the value function is bounded below by a power function.

Employing the standard viscosity theory, one demonstrates that the value function $V(x,t,h)$  defined in \eqref{value} stands as the sole viscosity solution to the following variational inequality, 
\begin{align}\label{V_pb}
\left\{
\begin{array}{ll}
\min\Big\{-V_t-\sup\limits_{\pi}\Big(\frac{1}{2}\si^2\pi^2 V_{x x } +(\mu-r)\pi V_x \Big)\\
\quad\quad\quad\quad
- \sup\limits_{bh \leq c\leq h}\Big(U(c)-c V_x \Big)-r x  V_x  +\rho V,\;-V_h\Big\}=0, \quad (x,t,h)\in \RR_T,\\
V(x ,h,T)=U_T (x ),\quad x>0,\;h>0,
\end{array}
\right.
\end{align}
where 
$$
\RR_T=\Big\{(x,t,h)\;\Big|\; x>\frac{b h}{r}(1-e^{-r(T-t)}),\;0\leq t< T,\;h>0\Big\}.
$$


\section{Main results}\label{sec:D}
\setcounter{equation}{0}

In this section, we present the main conclusions of the paper concerning the solution and free boundaries associated with Problem \eqref{V_pb}. By establishing the verification theorem, we derive key properties of the value function \eqref{value}, along with the corresponding optimal investment and consumption strategies.

Regarding the solvability of Problem \eqref{V_pb}, we have the following result:

\begin{theorem}\label{thm:V}
There exists a solution \( V \in C^{2,1,1}(\RR_T) \cap C(\RR_T \cup \{t = T\}) \) to Problem \eqref{V_pb} such that
\begin{align}\label{V}
 C_T \left( \frac{1}{1 - \theta} \left[ x - \frac{b h}{r} (1 - e^{-r(T - t)}) \right]^{1 - \theta} - 1 \right)
 \leq V(x, t, h) \leq C_T \left( \frac{1}{1 - \gamma} x^{1 - \gamma} + 1 \right), \\[2mm] \label{V1}
V_x > 0, \quad V_{xx} < 0, \\[2mm] \label{V2}
\lim_{x \to \frac{b h}{r} (1 - e^{-r(T - t)})} V_x = +\infty, \\[2mm] \label{V3}
\lim_{x \to +\infty} V_x = 0,
\end{align}
where \( 0 < \gamma < 1 \) and \( \theta > 1 \) are the parameters appearing in conditions \eqref{U2} and \eqref{U3}, and \( C_T \) is a positive constant depending only on \( T \).
\end{theorem}

The proof can be found in Section~\ref{sec:V}.

To analyze the structure of the variational inequality in \eqref{V_pb}, we partition the domain \( \RR_T \) into the following two regions:

\begin{itemize}
\item {\bf Switching region:}
\[
\SS = \{ (x, t, h) \in \RR_T \mid V_h = 0 \}. 
\]
This means that the consumption habit is being updated.
\item {\bf Non-switching region:}
\[
\NS = \{ (x, t, h) \in \RR_T \mid V_h < 0 \}. 
\]
This means that the consumption habit remains unchanged.
\end{itemize}

In economic terms, when \( (X_t, H_t, t) \in \SS \), the agent raises the consumption rate \( C_t \) beyond its historical maximum \( H_t \), thereby increasing the habit level. Conversely, when \( (X_t, H_t, t) \in \NS \), the agent maintains a consumption rate below \( H_t \), and the reference level remains constant.

Based on the properties of the value function established in Theorem~\ref{thm:V}, the optimal investment and consumption strategies, denoted \( \pi^*(x, t, h) \) and \( c^*(x, t, h) \) respectively, are given by:
\begin{align}
\pi^*(x,t,h) &= -\frac{\mu - r}{\sigma^2} \frac{V_x}{V_{xx}} > 0, \label{pi*} \\
c^*(x,t,h) &= \label{c*}
\begin{cases}
h, & \text{if } 0 < V_x \leq U'(h), \\
(U')^{-1}(V_x), & \text{if } U'(h) < V_x < U'(b h), \\
b h, & \text{if } V_x \geq U'(b h),
\end{cases}
\end{align}
where \( (U')^{-1}(\cdot) \) denotes the inverse of the marginal utility function \( U' \).

The non-switching region \( \NS \) can be further decomposed into three subregions:
\begin{itemize}
\item {\bf High consumption region:}
\[
\HC = \{ V_x \leq U'(h) \} \cap \NS,
\]
\item {\bf Medium consumption region:}
\[
\MC = \{ U'(b h) < V_x < U'(h) \} \cap \NS,
\]
\item {\bf Low consumption region:}
\[
\LC = \{ V_x \geq U'(b h) \} \cap \NS.
\]
\end{itemize}

These regions are separated by free boundaries, which are characterized in the following result:

\begin{theorem}\label{thm:h}
Assume \( V \) is the solution to Problem \eqref{V_pb} obtained in Theorem~\ref{thm:V}, and \( c^* \) is the corresponding optimal consumption policy given in \eqref{c*}. Then there exist three continuous functions \( x^L(t,h) \), \( x^H(t,h) \), and \( x^*(t,h) \), defined on \( [0, T) \times \mathbb{R}^+ \), satisfying
\[
\frac{b h}{r} (1 - e^{-r(T - t)}) < x^L(t,h) < x^H(t,h) < x^*(t,h) < +\infty,
\]
such that the domain \( \RR_T \) is partitioned into the following four regions:
\begin{align*}
\SS &= \{ (x,t,h) \in \RR_T \mid x \geq x^*(t,h) \}, \\
\NS &= \{ (x,t,h) \in \RR_T \mid \tfrac{b h}{r}(1 - e^{-r(T - t)}) < x < x^*(t,h) \}, \\
\HC &= \{ (x,t,h) \in \RR_T \mid x^H(t,h) \leq x < x^*(t,h) \}, \\
\MC &= \{ (x,t,h) \in \RR_T \mid x^L(t,h) < x < x^H(t,h) \}, \\
\LC &= \{ (x,t,h) \in \RR_T \mid \tfrac{b h}{r}(1 - e^{-r(T - t)}) < x \leq x^L(t,h) \}.
\end{align*}

Moreover, the boundary functions \( x^L(t,h) \), \( x^H(t,h) \), and \( x^*(t,h) \) are strictly increasing in \( h \), and satisfy:
\begin{align}
x^*(t, +\infty) &= +\infty, \label{x*H} \\
x^L(T^-, h) &= \begin{cases}
(U_T')^{-1}(U'(b h)), & \text{if } U'(b h) < U_T'(0+), \\
0, & \text{otherwise},
\end{cases} \label{xLT} \\
x^H(T^-, h) &= \begin{cases}
(U_T')^{-1}(U'(h)), & \text{if } U'(h) < U_T'(0+), \\
0, & \text{otherwise},
\end{cases} \label{xHT} \\
x^*(T^-, h) &= x^H(T^-, h). \label{x*T}
\end{align}
Here, we use the notation \( f(a\pm) = \lim_{x \to a^\pm} f(x) \) and \( f(\pm\infty) = \lim_{x \to \pm\infty} f(x) \).

(See Figure~\ref{fig}.)
\end{theorem}

\begin{figure}[H]
\begin{center}
\begin{picture}(200,100)(0,0)
\put(0,10){\vector(1,0){200}} \put(20,-15){\vector(0,1){135}}
\put(0,100){\line(1,0){200}}
\put(10,110){$t$}
\put(190,0){\footnotesize $x$}
\put(90,105){\footnotesize $x^H(T^-,h)=x^*(T^-,h)$}
\put(40,105){\footnotesize $x^L(T^-,h)$}
\put(170,70){\footnotesize $\SS$}
\put(120,70){\footnotesize $\HC$}
\put(70,70){\footnotesize $\MC$}
\put(30,70){\footnotesize $\LC$}
\put(130, 97){\footnotesize $\bullet$}
\put(55, 97){\footnotesize $\bullet$}
\qbezier(140,10)(165,78)(132,100)\put(133,45){\footnotesize $x^*(t,h)$}
\qbezier(100,10)(100,78)(132,100)\put(83,45){\footnotesize $x^H(t,h)$}
\qbezier(60,10)(50,70)(58,100)\put(33,45){\footnotesize $x^L(t,h)$}
\end{picture}
\caption{The locations of the four regions and the three free boundaries for fixed \( h > 0 \).}
\label{fig}
\end{center}
\end{figure}
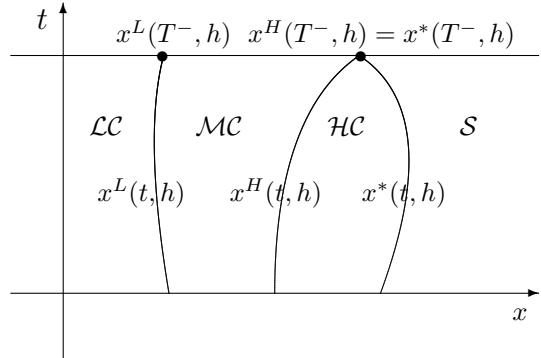

The verification theorem is stated as follows. To circumvent the technical complications arising from verifying convergence under the localization procedure for the Itô integral, we restrict our analysis to the class of bounded value functions—an assumption justified by the boundedness of the utility functions.

\begin{theorem}[Verification Theorem]
\label{thm:averi}
Let \( V \) be the solution to the Hamilton-Jacobi-Bellman equation \eqref{V_pb}, as characterized in Theorem~\ref{thm:V}. If \( V \) is bounded on \( \RR_T \), then it coincides with the value function of the optimal consumption-investment problem \eqref{value}.

Furthermore, let \( x^*(t,h) \) denote the optimal switching boundary defined in Theorem~\ref{thm:h}, and consider the controlled process \( \{(X^*_s, H^*_s, \Pi^*_s, C^*_s)\}_{t \leq s \leq T} \) governed by the following system:
\[
\left\{
\begin{array}{ll}
\mathrm{d} X^*_s = \big[r (X^*_s - \Pi^*_s) + \mu \Pi^*_s - C^*_s\big]\, \mathrm{d}s + \sigma \Pi^*_s\, \mathrm{d}W_s, \quad t \leq s \leq T, \\[1mm]
H^*_s = (x^*)^{-1}(s,\cdot)\left( \max\left\{ x^*(s,h),\; \max_{t \leq \theta \leq s} X^*_\theta \right\} \right), \\[1mm]
\Pi^*_s = \pi^*(X^*_s, s, H^*_s), \\[1mm]
C^*_s = c^*(X^*_s, s, H^*_s),
\end{array}
\right.
\]
where \( (x^*)^{-1}(s,\cdot) \) denotes the inverse function of \( x^*(s,\cdot) \), and the feedback control functions \( \pi^* \) and \( c^* \) are defined in \eqref{pi*} and \eqref{c*}, respectively. Then the strategy \( \{(\Pi^*_s, C^*_s)\}_{t \leq s \leq T} \) is optimal for Problem~\eqref{value}.
\end{theorem}

\begin{proof}
Let \( \widetilde{V}(x,t,h) \) denote the value function defined in \eqref{value}, and let \( V(x,t,h) \) be the solution to the PDE \eqref{V_pb}. We aim to prove that \( V(x,t,h) = \widetilde{V}(x,t,h) \).

Fix any \( (x,t,h) \in \mathbb{R}^+ \times [0,T) \times \mathbb{R}^+ \), and let \( \{(\Pi_s, C_s)\}_{t \leq s \leq T} \) be any admissible strategy. Define
\[
H_s = \max\left\{ h,\; \max_{t \leq \theta \leq s} C_\theta \right\}, \quad s \geq t,
\]
and let \( \{X_s\}_{s \geq t} \) be the solution to the wealth equation \eqref{Xt}.
For each \( n \in \mathbb{N} \), define the stopping time
\[
\tau_n := \inf \left\{ s \in [t, T] \;\middle|\; X_s + |\Pi_s| + |H_s| > n \text{ or } X_s < \frac{b h}{r} (1 - e^{-r(T - t)}) + \frac{1}{n} \right\}.
\]
Applying It\^o formula to \( V(X_s, s, H_s) \) on \( [t, \tau_n \wedge T] \) yields:
\begin{align*}
& V(x,t,h) =\mathbb{E} \left[ e^{-\rho(\tau_n \wedge T - t)} V(X_{\tau_n \wedge T}, \tau_n \wedge T, H_{\tau_n \wedge T}) \right] \\
& - \mathbb{E} \left[ \int_t^{\tau_n \wedge T} e^{-\rho(s - t)} \left( -V_t - \frac{1}{2} \sigma^2 \Pi_s^2 V_{xx} - (\mu - r) \Pi_s V_x + V_x C_s - r X_s V_x + \rho V \right)(X_s,s,H_s)\, \mathrm{d}s \right] \\
& - \mathbb{E} \left[ \int_t^{\tau_n \wedge T} e^{-\rho(s - t)} V_h(X_s,s,H_s)\, \mathrm{d}H_s \right] \\
& - \mathbb{E} \left[ \int_t^{\tau_n \wedge T} e^{-\rho(s - t)} \sigma \Pi_s V_x(X_s,s,H_s)\, \mathrm{d}W_s \right].
\end{align*}

The last term vanishes in expectation due to the integrability and boundedness of the integrand. From the variational inequality satisfied by \( V \), we have:
\begin{align*}
&\left( -V_t - \frac{1}{2} \sigma^2 \Pi_s^2 V_{xx} - (\mu - r) \Pi_s V_x + V_x C_s - r X_s V_x + \rho V \right)(X_s,s,H_s) \\
&\quad \geq \sup_{\pi, c} \left\{ -V_t - \frac{1}{2} \sigma^2 \pi^2 V_{xx} - (\mu - r) \pi V_x - r X_s V_x + \rho V - (U(c) - c V_x) \right\} + U(C_s) \\
&\quad \geq U(C_s),
\end{align*}
and since \( \mathrm{d}H_s \geq 0 \), we have \( -V_h(X_s,s,H_s) \mathrm{d}H_s \leq 0 \).
Hence,
\[
V(x,t,h) \geq \mathbb{E} \left[ \int_t^{\tau_n \wedge T} e^{-\rho(s - t)} U(C_s)\, \mathrm{d}s \right] + \mathbb{E} \left[ e^{-\rho(\tau_n \wedge T - t)} V(X_{\tau_n \wedge T}, \tau_n \wedge T, H_{\tau_n \wedge T}) \right].
\]
Since \( V \) is bounded, the dominated convergence theorem implies
\[
V(x,t,h) \geq \mathbb{E} \left[ \int_t^T e^{-\rho(s - t)} U(C_s)\, \mathrm{d}s + e^{-\rho(T - t)} U_T(X_T) \right].
\]
Because \( (\Pi_s, C_s) \) is arbitrary, we conclude \( V(x,t,h) \geq \widetilde{V}(x,t,h) \).

To establish the reverse inequality, consider the feedback strategy \( (\Pi^*_s, C^*_s) \) defined in the theorem. By construction,
\[
X^*_s \leq x^*(s, H^*_s), \quad \forall s \in [t, T],
\]
and the HJB equation holds with equality:
\[
\left( -V_t - \frac{1}{2} \sigma^2 (\Pi^*_s)^2 V_{xx} - (\mu - r) \Pi^*_s V_x + V_x C^*_s - r X^*_s V_x + \rho V \right)(X^*_s, s, H^*_s) = U(C^*_s).
\]
Moreover, since
\[
\mathrm{d}H^*_s = 0 \quad \text{if } X^*_s < x^*(s, H^*_s), \qquad V_h(X^*_s, s, H^*_s) = 0 \quad \text{if } X^*_s = x^*(s, H^*_s),
\]
it follows that
\[
-V_h(X^*_s, s, H^*_s)\, \mathrm{d}H^*_s = 0, \quad \forall s \in [t, T].
\]
Substituting into Itô’s formula and passing to the limit yields
\[
V(x,t,h) = \mathbb{E} \left[ \int_t^T e^{-\rho(s - t)} U(C^*_s)\, \mathrm{d}s + e^{-\rho(T - t)} U_T(X^*_T) \right] = \widetilde{V}(x,t,h).
\]
Thus, \( V = \widetilde{V} \), and the feedback strategy \( (\Pi^*_s, C^*_s) \) is indeed optimal.
\end{proof}

\section{Simplification of the Equation}\label{sec:S}
\setcounter{equation}{0}

To analyze the fully nonlinear three-dimensional variational inequality \eqref{V_pb}, we apply a dual transformation that reduces the problem to a linear obstacle problem \eqref{w_pb}. In this reformulation, the variable $h$ is treated as a parameter, thereby reducing \eqref{w_pb} to a two-dimensional problem. As a result, classical PDE methods can be employed to study both the solution and the associated free boundary.

Throughout this procedure, various a priori estimates on the solution will be used. We will rigorously establish the existence of solutions to Problem \eqref{w_pb}, and in Section~\ref{sec:V}, we construct a solution to Problem \eqref{V_pb} based on the solution to \eqref{w_pb}, while verifying the necessary a priori conditions.

Given the optimal feedback functions in \eqref{pi*} and \eqref{c*}, the variational inequality \eqref{V_pb} can be reformulated as
\begin{align}\label{V_pb1}
\left\{
\begin{array}{ll}
\displaystyle{\min\Bigg\{-V_t + \frac{1}{2} \kappa^2 \frac{V_x^2}{V_{xx}} - \wU(V_x,h) - r x V_x + \rho V,\; -V_h\Bigg\} = 0}, \quad (x,t,h) \in \RR_T, \\[2mm]
V(x,h,T) = U_T(x), \quad x > 0,\; h > 0,
\end{array}
\right.
\end{align}
where \( \kappa = (\mu - r)/\sigma > 0 \), and the function \( \wU(y,h) \) is defined by
\[
\wU(y,h) := \max_{b h \leq c \leq h} \big( U(c) - c y \big) =
\begin{cases}
U(h) - h y, & 0 < y \leq U'(h), \\
U((U')^{-1}(y)) - (U')^{-1}(y) y, & U'(h) < y < U'(b h), \\
U(b h) - b h y, & y \geq U'(b h).
\end{cases}
\]

Assuming the following priori estimates:
\[
\lim_{x \to \frac{b h}{r} (1 - e^{-r(T - t)})^+} V_x(x,t,h) = +\infty, \quad \lim_{x \to +\infty} V_x(x,t,h) = 0, \quad t \in [0,T),\; h > 0,
\]
we define the dual transform of \( V(x,t,h) \) as
\[
v(y,\tau,h) := \sup_{x > \frac{b h}{r}(1 - e^{-r(T - t)})} \left( V(x,t,h) - x y \right), \quad y > 0,\; h > 0,\; \tau = T - t \in (0,T].
\]

For fixed \( x, t, h \), let \( y := V_x(x,t,h) \). Then it follows that
\[
V(x,t,h) = v(y,\tau,h) + x y.
\]
By the convexity of \( v(y,\tau,h) \), we obtain the relations:
\[
x = -v_y(y,\tau,h), \quad
V(x,t,h) = v(y,\tau,h) - y v_y(y,\tau,h),
\]
\[
V_{xx}(x,t,h) = -\frac{1}{v_{yy}(y,\tau,h)}, \quad
V_t(x,t,h) = -v_\tau(y,\tau,h), \quad
V_h(x,t,h) = v_h(y,\tau,h).
\]
Substituting into \eqref{V_pb1}, we obtain the transformed variational inequality:
\begin{align}\label{v_pb}
\left\{
\begin{array}{ll}
\min \left\{ v_\tau - \LL v - \wU(y,h),\; -v_h \right\} = 0, \quad (y, \tau, h) \in \mathbb{R}^+ \times (0,T] \times \mathbb{R}^+, \\[2mm]
v(y, h, 0) = \wwU_T(y), \quad y > 0,\; h > 0,
\end{array}
\right.
\end{align}
where the operator \( \LL \) is given by
\[
\LL v := \frac{1}{2} \kappa^2 y^2 v_{yy} + (\rho - r) y v_y - \rho v,
\]
and the function \( \wwU_T(y) \) denotes the Legendre transform of the terminal utility \( U_T \), i.e.,
\begin{align}\label{wUT}
\wwU_T(y) := \sup_{x > 0} \big( U_T(x) - x y \big) =
\begin{cases}
U_T((U_T')^{-1}(y)) - (U_T')^{-1}(y) y, & 0 < y < U_T'(0+), \\
U_T(0), & y \geq U_T'(0+).
\end{cases}
\end{align}

To facilitate further analysis, we introduce the logarithmic transformation
\[
z = \ln y, \quad u(z, \tau, h) = v(y, \tau, h),
\]
under which the operator \( \LL \) becomes
\[
\TT u := \frac{1}{2} \kappa^2 u_{zz} + \left( \rho - r - \frac{1}{2} \kappa^2 \right) u_z - \rho u.
\]
Thus, the variational inequality for \( u(z, \tau, h) \) becomes:
\begin{align}\label{u_pb}
\left\{
\begin{array}{ll}
\min \left\{ u_\tau - \TT u - \wU(e^z, h),\; -u_h \right\} = 0, \quad (z, \tau, h) \in \mathbb{R} \times (0,T] \times \mathbb{R}^+, \\[2mm]
u(z, h, 0) = \wwU_T(e^z), \quad z \in \mathbb{R},\; h \in \mathbb{R}^+.
\end{array}
\right.
\end{align}

Define the function
\[
f(z, h) := \partial_h \wU(e^z, h) =
\begin{cases}
U'(h) - e^z, & \text{if } e^z \leq U'(h), \\
0, & \text{if } U'(h) < e^z < U'(b h), \\
b U'(b h) - b e^z, & \text{if } e^z \geq U'(b h).
\end{cases}
\]

In the region where \( u_\tau - \TT u - \wU(e^z, h) = 0 \), the function \( w := -u_h \) satisfies:
\[
w_\tau - \TT w + f(z, h) = -\partial_h(u_\tau - \TT u - \wU(e^z, h)) = 0.
\]
This motivates the study of the following obstacle problem for \( w \):
\begin{align}\label{w_pb}
\left\{
\begin{array}{ll}
\min \left\{ w_\tau - \TT w + f(z,h),\; w \right\} = 0, \quad (z, \tau) \in \mathbb{R} \times (0,T],\; h \in \mathbb{R}^+, \\[2mm]
w(z, h, 0) = 0, \quad z \in \mathbb{R},\; h \in \mathbb{R}^+.
\end{array}
\right.
\end{align}

We begin our analysis by examining Problem~\eqref{w_pb} to understand the structure of the solution and its associated free boundary, which in turn provides the foundation for constructing the solution to the original problem \eqref{V_pb}.

\section{The solution and the free boundary to the obstacle problem \eqref{w_pb}}\label{sec:E}
\setcounter{equation}{0}

In this section, we will employ the penalty approximation method to establish the existence and uniqueness of the solution to the obstacle problem \eqref{w_pb} and derive certain estimates of the solution. Subsequently, we can prove the existence and uniqueness of the free boundary and elucidate its properties, such as monotonicity, smoothness, and the position of its starting point.

\subsection{Existence and uniqueness of solution to Problem \eqref{w_pb}}

The problem \eqref{w_pb} is essentially a two - dimensional problem as the variable $h$ can be treated as a parameter. Therefore, in this subsection, we will fix a positive value for $h$.
To address Problem \eqref{w_pb}, we initially formulate the following penalty approximation problem for it within a bounded domain.
\begin{align}\label{we_pb}
\left\{
\begin{array}{ll}
w^\ep_\tau-\TT w^\ep + f(z,h) +\beta_\ep(w^\ep)=0,\quad(y,\tau)\in \Q_T^\ep,\\
w^\ep(-\frac{1}{\ep},\tau,h) =0,  \quad  \tau\in (0,T],\\
w^\ep_x(\frac{1}{\ep},\tau,h) = 0,  \quad   \tau\in (0,T],\\
w^\ep(z,0,h) =0, \quad   z\in (-\frac{1}{\ep},\frac{1}{\ep}),
\end{array}
\right.
\end{align}
where, the domain $\Q_T^\ep$ is defined as
$$\Q_T^\ep:=(-\frac{1}{\ep},\frac{1}{\ep})\times(0,T],$$
and $\beta _{\ep }(\cdot)$ is any penalty function satisfying the following properties
\begin{eqnarray*}
&&\beta _{\ep }(\cdot)\in C^{2}(-\infty,+\infty ), \quad\beta _{\ep}(0)=-U'(h)<0, \quad\beta _{\ep}(x)=0\; \text{ for } x\geq \ep, \\
&& \beta _{\ep
}(\cdot)\leq 0, \quad \beta _{\ep }^{\prime }(\cdot)\geq 0,\quad \beta _{\ep
}^{\prime \prime }(\cdot)\leq 0, \quad\lim\limits_{\ep \rightarrow 0} \beta _{\ep }(x)=\left\{
\begin{array}{ll}
0, & \mbox{if } x>0,\vspace{2mm} \\
-\infty, & \mbox{if } x<0,
\end{array}
\right.
\end{eqnarray*}%
\figref{fig1} demonstrates such a penalty function $\beta _{\ep }(\cdot)$.
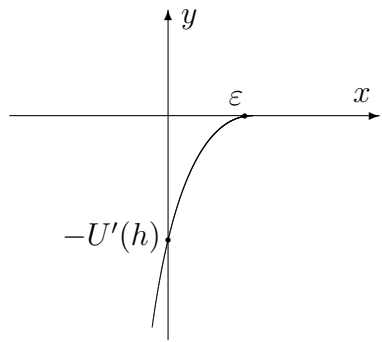
\begin{figure}[H]
\begin{center}
\begin{picture}(250,140)
\put(50,80){\vector(1,0){140}}
\put(110,-5){\vector(0,1){125}}
\qbezier(104,0)(115,80)(142,80)
\put(180,85){$x$}
\put(115,115){$y$}
\put(133,84){$\ep$}
\put(70,31){$-U'(h)$}
\put(110,33){\circle*{2}}
\put(139,80){\circle*{2}}
\end{picture} \caption{A penalty function $\beta_\ep (\cdot)$.} \label{fig1}
\end{center}
\end{figure}

Subsequently, we are able to derive the following existence and uniqueness results for the approximation problem \eqref{we_pb}.
\begin{lemma}\label{thm:we}
For any $p>3$ and $0<\al<1$, the problem \eqref{we_pb} admits a unique solution $w^\ep\in W^{2,1}_p(\Q_T^\ep)\cap C^{1+\al,\frac{1+\al}{2}}(\overline{\Q_T^\ep})$. Furthermore, the solution satisfies the following estimations in $\Q_T^\ep$:
\begin{align}\label{we1}
0\leq &\;w^\ep\leq \frac{1}{r} e^z+\ep,\\
\label{we2}
&\; w^\ep_z \geq 0,\\
\label{we3}
&\; w^\ep_\tau \geq 0.
\end{align}
\end{lemma}
\begin{proof}
The existence and uniqueness of the solution can be established by applying the Schauder fixed - point theorem (see \cite{Ev16} Theorem 2 on page 502) and the comparison principle. Since the proof process follows a standard procedure, we omit the details here.

We start by proving the inequality \eqref{we1}. Denote the function $\phi(z,\tau)\equiv0$. Then, we can calculate
\begin{align*}
\phi_\tau-\TT \phi+f(z,h)+\beta_\ep(\phi)= f(z,h) - U'(h)\leq 0.
\end{align*}
According to the comparison principle, we conclude that $w^\ep\geq \phi =0$, which is the left - hand side inequality in \eqref{we1}.

Next, denote the function $\Phi(z,\tau)=e^z/r +\ep$. Notice that $\Phi(z,\tau)\geq \ep$ implies $\beta_\ep(\Phi)=0$. Then, we have
\begin{align*}
\Phi_\tau-\TT \Phi + f(z,h) +\beta_\ep(\Phi)
=e^z + \rho \ep + f(z,h)\geq 0.
\end{align*}
Applying the comparison principle, we obtain $w^\ep\leq \Phi$, which verifies the right - hand side inequality in \eqref{we1}.

Now, we turn to the proof of \eqref{we2}. Differentiating the equation in \eqref{we_pb} with respect to $z$, we obtain
$$
\p_\tau w^\ep_z -\TT w^\ep_z + \beta_\ep'(w^\ep)w^\ep_z = - \p_z f(z,h) \geq 0.
$$
Moreover, since $w^\ep\geq 0$ and $w^\ep(-1/\ep,\tau,h)=0$, we can infer that $w^\ep_z(-1/\ep,\tau,h)\geq 0$ for all $\tau\in(0,T]$. Then, by applying the comparison principle, we conclude that $w^\ep_z\geq 0$.

Finally, we prove \eqref{we3}. For any fixed $0<\Delta\tau<T$, we define $\wt{w}^\ep(z,\tau,h):=w^\ep(z,\tau+\Delta\tau,h)$. Note that $\wt{w}^\ep(z,0,h)=w^\ep(z,\Delta\tau,h)\geq 0$, and $\wt{w}^\ep$ satisfies the same equation and boundary conditions as $w^\ep$ in $Q_{T-\Delta\tau}^\ep$. Therefore, by using the comparison principle, we can deduce that $\wt{w}^\ep\geq w^\ep$ in $Q_{T-\Delta\tau}^\ep$, which implies \eqref{we3}.
\end{proof}

\begin{proposition}\label{pro:we4}
$w^\ep$ is non decreasing with respect to $h$.
\end{proposition}
\begin{proof}
Since $f(z,h)$ is non decreasing with respect to $h$, we have for $h'>h$,
$$
w^\ep_\tau (z,\tau,h')-\TT w^\ep (z,\tau,h') +\beta_\ep(w^\ep(z,\tau,h')) = - f(z,h') \geq - f(z,h),
$$
so $w^\ep(z,\tau,h')$ is a super solution of \eqref{we_pb}, i.e. $w^\ep(\cdot,\cdot,h')\geq w^\ep(\cdot,\cdot,h)$.
\end{proof}

\begin{proposition}\label{pro:we4'}
We have for $h'>h>0$,
$$
w^\ep (z,\tau,h')\leq w^\ep (z,\tau,h) + \frac{1}{\rho} K(h,h')\quad (z,\tau)\in \Q^\ep,
$$
where
$$
K(h,h'):=\max\Big\{U'(h)-U'(h'),\;b U'(b h)-  b U'(b h')\Big\}.
$$
\end{proposition}
\begin{proof}
It is not hard to check that
$$f(z,h)-f(z,h')\leq K(h,h').$$
Let $\psi(z,\tau):=w^\ep (z,\tau,h) + K(h,h')/\rho$, then it satisfies that
$$
\psi_\tau (z,\tau)-\TT \psi (z,\tau) +\beta_\ep(\psi (z,\tau)) \geq - f(z,h)+ K(h,h')\geq - f(z,h').
$$
So $\psi(z,\tau)$ is a super solution of Problem \eqref{we_pb}, i.e. $\psi(z,\tau)\geq w^\ep(\cdot,\cdot,h')$.
\end{proof}

Based on the above lemmas, let $\varepsilon\rightarrow0$, via the embedding theorem, we then obtain the following existence and uniqueness of solution to Problem \eqref{w_pb}.
\begin{theorem}\label{thm:w}
The problem \eqref{w_pb} has a unique solution $w\in C^{1+\al,\frac{1+\al}{2}}(\R\times[0,T])\bigcap W^{2,1}_{p,\;\mathrm{loc}}(\R\times(0,T])$, for any $p>3$ and $0<\al<1$. Furthermore,
\begin{align}\label{w1}
0\leq &\;w\leq \frac{1}{r} e^z,\\ \label{w2}
&\;w_z \geq 0,\\ \label{w3}
&\;w_\tau \geq 0,\\ \label{w4}
0\leq&\; w_h \leq { \frac{1}{\rho}\max\Big\{-U''(h), -b^2 U''(bh)\Big\}.}
\end{align}
\end{theorem}

\subsection{Properties of the Free Boundary in Problem \eqref{w_pb}}

In the previous subsection, we established the existence and uniqueness of the solution to Problem~\eqref{w_pb}. Since the solution \( w \) is independent of the terminal time \( T \), we may, without loss of generality, extend the time variable to \( \tau \in (0, +\infty) \) in the analysis below.

Thanks to the monotonicity result \eqref{w2}, we define the free boundary of Problem~\eqref{w_pb} as
\begin{align}\label{eq:ytau}
z^*(\tau,h) := \sup \left\{ z \in \mathbb{R} \,\big|\, w(z,\tau,h) = 0 \right\}, \quad \tau, h \in \mathbb{R}^+.
\end{align}
Moreover, from inequalities \eqref{w3}--\eqref{w4}, it follows that \( z^*(\tau,h) \) is non-increasing in both \( \tau \) and \( h \).

We now demonstrate that the free boundary remains finite under all admissible inputs.

\begin{proposition}\label{pro:y_finite}
For any fixed \( h > 0 \), the limit
\[
z^*(+\infty,h) := \lim_{\tau \to +\infty} z^*(\tau,h)
\]
exists and is finite, i.e., \( z^*(+\infty,h) > -\infty \).
\end{proposition}

\begin{proof}
Since \( w \) is non-decreasing in \( \tau \) and uniformly bounded above for each fixed \( z \) and \( h \), there exists a limiting function \( w^\infty(z,h) := \lim_{\tau \to +\infty} w(z,\tau,h) \).

Assume for contradiction that \( w^\infty(z,h) > 0 \) for all \( z \in \mathbb{R} \). Then, by continuity and weak convergence,
\begin{align}\label{Luinfty}
- \TT w^\infty(z,h) + f(z,h) = - \lim_{\tau \to +\infty} w_\tau(z,\tau,h) \leq 0 \quad \text{in } \mathbb{R}.
\end{align}

Now consider \( z_0 < \ln(\tfrac{1}{2} U'(h)) \). Since \( f(z,h) = U'(h) - e^z \), we have \( f(z,h) > \tfrac{1}{2} U'(h) \) and therefore \( -\TT w^\infty(z,h) \leq -\tfrac{1}{2} U'(h) \) for all \( z < z_0 \).

But because \( w^\infty \geq 0 \) and \( w^\infty_z \geq 0 \), we obtain
\[
\frac{\kappa^2}{2} w^\infty_{zz}(z,h) + \rho w^\infty_z(z,h) \geq \frac{1}{2} U'(h), \quad z < z_0.
\]
Integrating both sides from \( y \) to \( y+1 \), and then integrating in \( y \) from \( N \) to \( N+1 \), yields
\begin{multline*}
\frac{\kappa^2}{2} \big(w^\infty(N+2,h) - 2w^\infty(N+1,h) + w^\infty(N,h)\big) \\
+ \rho \int_N^{N+1} \big(w^\infty(y+1,h) - w^\infty(y,h)\big) \mathrm{d}y \geq \frac{1}{2} U'(h).
\end{multline*}
However, the left-hand side tends to zero as \( N \to -\infty \), since \( w^\infty \) is non-decreasing and bounded, leading to a contradiction.
\end{proof}

\begin{proposition}\label{pro:zup}
For each \( h > 0 \), the free boundary satisfies the upper bound
\[
z^*(\tau,h) \leq \ln(U'(h)).
\]
\end{proposition}

\begin{proof}
We show that \( w(z,\tau,h) > 0 \) for all \( z > \ln(U'(h)) \). In this region, \( f(z,h) = U'(h) - e^z < 0 \), and hence
\[
w_\tau(z,\tau,h) - \TT w(z,\tau,h) \geq -f(z,h) > 0.
\]
Since \( w(z,\tau,h) \geq 0 \) and not identically zero in \( (\ln(U'(h)), +\infty) \), the strong maximum principle implies \( w(z,\tau,h) > 0 \) in this domain.
\end{proof}

\begin{proposition}\label{pro:y3}
The free boundary \( z^*(\tau,h) \) is continuous in \( (\tau, h) \in \mathbb{R}^+ \times \mathbb{R}^+ \), and satisfies
\[
\lim_{\tau \to 0+} z^*(\tau,h) = \ln(U'(h)).
\]
\end{proposition}

\begin{proof}
Suppose by contradiction that \( z^* \) is discontinuous at some point \( (\tau_0, h_0) \). Then there exists a sequence \( (\tau_n, h_n) \to (\tau_0, h_0) \) such that \( z^*(\tau_n, h_n) \not\to z^*(\tau_0, h_0) \). Denote
\[
\underline{z} = \min\big\{\lim z^*(\tau_n,h_n),\, z^*(\tau_0,h_0)\big\}, \quad
\overline{z} = \max\big\{\lim z^*(\tau_n,h_n),\, z^*(\tau_0,h_0)\big\}.
\]
Then for any \( z \in (\underline{z}, \overline{z}) \), we have \( w(z,\tau_0,h_0) = 0 \) and
\[
w_\tau(z,\tau_0,h_0) - \TT w(z,\tau_0,h_0) + f(z,h_0) = 0.
\]
However, since \( f(z,h_0) > 0 \) in this interval and \( w_\tau \geq 0 \), we obtain
\[
w_\tau + f > 0,
\]
contradicting the equation above. Thus \( z^*(\tau,h) \) is continuous.

To show \( \lim_{\tau \to 0+} z^*(\tau,h) = \ln(U'(h)) \), note that at \( \tau = 0 \), the initial condition \( w(z,0,h) = 0 \), and since \( f(z,h) > 0 \) when \( z < \ln(U'(h)) \), we must have \( z^*(\tau,h) \to \ln(U'(h)) \) as \( \tau \to 0+ \).
\end{proof}

\begin{proposition}\label{pro:y4}
The free boundary \( z^*(\tau,h) \) is strictly decreasing with respect to \( \tau \).
\end{proposition}

\begin{proof}
Assume for contradiction that \( z^*(\tau,h_0) = z_0 \) for all \( \tau \in [\tau_1, \tau_2] \). Then
\[
w(z_0,\tau,h_0) = 0, \quad w_z(z_0,\tau,h_0) = 0, \quad w_{z\tau}(z_0,\tau,h_0) = 0.
\]
Since \( w \) solves \( w_\tau - \TT w + f = 0 \) in \( z > z_0 \), we have $w_\tau$ that solves a linear equation with zero initial gradient at the boundary. By the Hopf lemma, we must have \( w_{\tau z}(z_0,\tau,h_0) > 0 \), which contradicts the fact that \( w_{\tau z} = 0 \). Thus, \( z^*(\tau,h) \) is strictly decreasing in \( \tau \).
\end{proof}

\begin{proposition}\label{pro:y5}
{ For each fixed \( h > 0 \), the function \( w_\tau(\cdot, \cdot, h) \in C(\mathbb{R}\times \R^+) \), and the free boundary \( z^*(\cdot,h) \in C^\infty(\mathbb{R}^+) \).}
\end{proposition}
{
\begin{proof}
Since $w_\tau \geq0$, $w\geq0$, and $f(\cdot,h)\in C^\infty$ near $z=z^*(\tau, h)$, we can take advantage of the same arguments in \cite{Fr75} to prove $w_\tau(\cdot,\cdot,h) $ is continuous across $z=z^*(\tau, h)$  and $z^*(\cdot, h)\in C^{\infty}(\R^+)$.
\end{proof}
}

\begin{proposition}\label{pro:y6}
The free boundary \( z^*(\tau,h) \) is strictly decreasing with respect to \( h \).
\end{proposition}

\begin{proof}
Suppose there exists \( h \in [h_1, h_2] \) such that \( z^*(\tau_0, h) = z_0 \) for all \( h \in [h_1, h_2] \). Then,
\[
w_z(z_0,\tau_0,h_1) = w_z(z_0,\tau_0,h_2) = 0.
\]
Let \( \psi(z,\tau) := w(z,\tau,h_2) - w(z,\tau,h_1) \).
{
Since $z^*(\tau,h_1)\geq z^*(\tau,h_2)$, the equation of $w$ gives
\[
\psi_\tau (z,\tau)-\TT \psi (z,\tau)=f(z,h_1)-f(z,h_2)\geq 0, \quad z\geq z^*(\tau,h_1).
\]
Since $\psi(z,\tau)\geq 0$ for $z\geq z^*(\tau,h_1)$ and $\psi(z_0,\tau_0)= 0$, and noting that $z^*(\cdot,h_1)\in C^\infty(\R^+)$, we adopt the Hopf Lemma to get $\psi_z(z_0,\tau_0)>0$, contradicting the previous conclusion. Therefore, \( z^*(\tau,h) \) is strictly decreasing in \( h \).
}
\end{proof}

To construct the solution to Problem~\eqref{u_pb}, we define
\[
h^*(z,\tau) := \sup \{ h \geq 0 \mid w(z,\tau,h) = 0 \}, \quad (z,\tau) \in \mathbb{R} \times (0,T].
\]
Then
\begin{align}\label{h*1}
h^*(z,\tau) =
\begin{cases}
(z^*(\tau, \cdot))^{-1}(z), & \text{if } z < z^*(\tau,0+), \\
0, & \text{if } z \geq z^*(\tau,0+).
\end{cases}
\end{align}

{ Since $z^*(\cdot,\cdot)$ is continuous and non increasing with respect to $\tau$ and $h$, we have}
\begin{proposition}\label{pro:y7}
The function \( h^*(z,\tau) \) is continuous and non-increasing in both \( z \) and \( \tau \).
\end{proposition}

\section{The Solution and the Free Boundary to the Dual Problem}\label{sec:E1}
\setcounter{equation}{0}

In this section, we aim to construct the solution to the dual problem \eqref{v_pb}, or its equivalent form \eqref{u_pb}, by utilizing the solution to the obstacle problem \eqref{w_pb}. Since $w = -u_h$, it is natural to obtain $u$ by integrating $w$ with respect to $h$. However, it is not feasible to identify a specific point $h_0$ to determine $u(\cdot,\cdot,h_0)$ or to evaluate $u(\cdot,\cdot,+\infty)$ as the initial value of the integration.

To overcome this difficulty, we introduce a new model that enhances the original by imposing a constraint on the consumption rate $C_t$ (or, equivalently, the maximum consumption rate $H_t$), restricting it to a bounded interval $(h,\hh]$. Under this restriction, when $H_t$ reaches the upper bound $\hh$, it remains fixed at this level. The dual problem reduces to a linear two-dimensional problem, whose solution can be determined by solving a linear PDE. This solution can then serve as the initial value for the abovementioned integration to solve this new problem. Then, by taking the limit $\hh \to \infty$, we eventually obtain the solution to the original problem \eqref{v_pb}, which imposes no upper bound on $h$.

We now turn our attention to the following linear problem for the function $\fh(z,\tau)$:
\begin{align}\label{fh_pb}
\left\{\begin{array}{ll}
\p_\tau \fh - \TT \fh - \wU (e^z, \hh) = 0, & (z,\tau) \in \R \times (0,T], \\
\fh(z,0) = \wwU_T (e^z), & z \in \R.
\end{array}\right.
\end{align}

According to standard PDE theory, this problem admits a unique solution $\fh \in C^{2,1}(\R \times [0,T])$ under the exponential growth condition. The solution to Problem \eqref{fh_pb} corresponds to an optimal control problem where the control variable $h$ is fixed at the constant level $\hh$.

Next, we consider the following approximation problem for \eqref{u_pb}:
\begin{align}\label{uh_pb}
\left\{\begin{array}{ll}
\min\Big\{\p_\tau \uh - \TT \uh - \wU(e^z,h),\;  -\p_h\uh\Big\} = 0, & (z,\tau,h) \in \R \times (0,T] \times (0,\hh), \\
\uh(z,0,h) = \wwU_T(e^z), & z \in \R,\; h \in (0,\hh), \\
\uh(z,\tau,\hh) = \fh(z,\tau), & (z,\tau) \in \R \times (0,T].
\end{array}\right.
\end{align}
The solution to Problem \eqref{uh_pb} represents an optimal control scenario where the control variable $h$ is restricted to the bounded interval $(0,\hh]$.

With solutions to Problems \eqref{fh_pb} and \eqref{w_pb} in hand, we are now in a position to construct the solution to Problem \eqref{uh_pb}. Suppose $\fh$ solves \eqref{fh_pb} and $w$ solves \eqref{w_pb}. Define
\begin{align}\label{uw}
\uh(z,\tau,h) := \fh(z,\tau) + \int_h^\hh w(z,\tau,s) \, \mathrm{d}s.
\end{align}

According to Theorem~\ref{thm:w} and Proposition~\ref{pro:y5}, we know that $w_z(\cdot,\cdot,h),\; w_\tau(\cdot,\cdot,h) \in C(\R \times (0,T])$ for every $h \in (0,\hh]$, and $w \in C(\R \times (0,T] \times \R^+)$. Moreover, we can easily find super-solutions for the PDEs involving $w_z$ and $w_\tau$ that are independent of $h$. As a result, we can apply the Lebesgue dominated convergence theorem to deduce that $\uh \in C^{1,1,1}(\R \times [0,T] \times (0,\hh])$. Additionally, we have the following result:

\begin{lemma}\label{lem:TuhR}
The function $\uh$ defined in \eqref{uw} belongs to $C^{2,1,1}(\R \times [0,T] \times (0,\hh])$ and solves Problem \eqref{uh_pb}.
\end{lemma}

\begin{proof}
The initial and boundary conditions in \eqref{uh_pb} are satisfied by construction. We now verify the variational inequality in \eqref{uh_pb}.

First, we note that $-\p_h \uh = w \geq 0$.

Second, observe that
\[
(w_\tau - \TT w)(z,\tau,h) =
\begin{cases}
-f(z,h) = -\p_h \wU(e^z,h), & h > h^*(z,\tau), \\
0, & h < h^*(z,\tau).
\end{cases}
\]
Applying \eqref{uw}, we obtain
\begin{align}
(\uh_\tau -\TT \uh)(z,\tau,h) =&\; (\fh_\tau -\TT \fh)(z,\tau)+ \int_h^{\hh} (w_\tau -\TT w)(z,\tau,s) \d s \nonumber\\ \label{Luh}
= &\; \wU(e^z,\hh) + \int_{h^*(z,\tau) \vee h}^{\hh} -\p_h\wU(e^z,s) \d s = \wU (e^z, h^*(z,\tau)\vee h).
\end{align}
Since $h^* \in C(\R \times (0,T])$ and $\uh \in C^{1,1,1}(\R \times (0,T] \times (0,\hh])$, we conclude that $\uh \in C^{2,1,1}(\R \times (0,T] \times (0,\hh])$.

If $-\p_h \uh(z,\tau,h) = w(z,\tau,h) > 0$, i.e., $h > h^*(z,\tau)$, then
\[
(\uh_\tau - \TT \uh)(z,\tau,h) = \wU(e^z,h).
\]
Conversely, if $-\p_h \uh(z,\tau,h) = w(z,\tau,h) = 0$, then $h \leq h^*(z,\tau)$, or equivalently, $z \leq z^*(\tau,h) \leq \ln U'(h)$. Since $\wU(e^z,h)$ is non-decreasing with respect to $h$ when $z \leq \ln U'(h)$, it follows that
\[
(\uh_\tau - \TT \uh)(z,\tau,h) = \wU(e^z,h^*(z,\tau)) \geq \wU(e^z,h).
\]
Therefore, the variational inequality in \eqref{uh_pb} is satisfied.
\end{proof}

In order to obtain the solution of \eqref{u_pb} by taking the limit of $\uh$ for $\hh\to +\infty$, we must find an upper bound of $\uh$ that is independent of $\hh$.
Define
\begin{align}\label{wwU}
\wwU(y)=\sup\limits_{c>0} \Big(U(c)-c y\Big)
\end{align}
and assume that $\uu(z,\tau)$ is the unique $C^{2,1}(\R \times[0,T])$ solution of
\begin{align}\label{uu_pb}
\left\{\begin{array}{ll}
\p_\tau \uu - \TT \uu-\wwU (e^z)=0,\quad (z,\tau)\in \R \times(0,T],\\
\uh(z,0)= \wwU_T (e^z),\quad z\in \R.
\end{array}\right.
\end{align}
under the exponential growth condition.
This problem corresponds to the model where there is no drawdown constraint on consumption.
Since
\begin{align}\label{wwUwU}
\wwU(y)\geq \wU (y,\hh),
\end{align}
by the comparison principle we have
\begin{align}\label{uhuh}
\uu \geq \fh.
\end{align}
Moreover, we have
\begin{lemma}\label{lem:uu}
For any $\hh\in\R^+$, there exists
\begin{align}\label{uu}
\uh(z,\tau,h) \leq \uu(z,\tau),\quad (z,\tau,h)\in \R\times[0,T]\times(0,\hh].
\end{align}
\end{lemma}
\begin{proof}
For any $\underline{h}\in(0,\hh)$, define
$$\zeta(z,\tau)=e^{\la z+ \eta \tau}+e^{-\la z+ \eta \tau},\quad (z,\tau)\in \R \times(0,T],$$
where $\la>0$ and $\eta>0$ are chosen to be large enough such that
$$
\lim\limits_{z\to \pm\infty}\max\limits_{(\tau,h)\in [0,T]\times [\underline{h},\hh]}\frac{|\uh(z,\tau,h)|+|\uu(z,\tau)|}{\zeta(z,\tau)}=0
$$
and
\begin{align}\label{Lphi}
\zeta_\tau -\TT \zeta \geq  0,\quad (z,\tau)\in \R \times(0,T].
\end{align}
We come to prove that for any $\ep>0$,
\begin{align}\label{v<=w}
\uh(z,\tau,h) \leq \uu(z,\tau)+\ep \zeta,\quad (z,\tau)\in \R \times(0,T],\;h\in [\underline{h},\hh].
\end{align}
We prove this by contradiction.  Suppose that \eqref{v<=w} is not true, let
$$M=\sup\limits_{(z,\tau,h)\in \R\times[0,T]\times [\underline{h},\hh]}(\uh-\uu-\ep \zeta)(z,\tau,h),$$
then $M>0$.
Note that $\uh-\uu-\ep \zeta$ is continuous and tends to $-\infty$ when $z\to\pm\infty$, and $(\uh-\uu-\ep\zeta)(z,0,h)=-\ep \zeta(z,0)<0$,
moreover, by \eqref{uhuh}, we further have
$(\uh-\uu-\ep\zeta)(z,\tau,\hh)\leq -\ep \zeta(z,\tau)<0$.
So $\uh-\uu-\ep\zeta$ attains its maximum value at a point $(z_0,\tau_0,h_0)\in \R\times(0,T]\times[\underline{h},\hh)$,
thus
$$
\p_\tau (\uh-\uu-\ep\zeta)(z_0,\tau_0,h_0)\geq 0,\quad \p_z (\uh-\uu-\ep\zeta)(z_0,\tau_0,h_0) = 0,\quad \p_{zz} (\uh-\uu-\ep\zeta)(z_0,\tau_0,h_0) \leq 0.
$$
then we get
\begin{align}\label{Luu>0}
\Big(\p_\tau (\uh-\uu-\ep\zeta)-\TT (\uh-\uu-\ep\zeta)\Big)(z_0,\tau_0,h_0) \geq \rho M >0.
\end{align}
On the other hand, we assume that $h_0$ satisfies
\begin{align*}
(\uh-\uu-\ep\zeta)(z_0,\tau_0,h)<M, \quad\forall\; h\in (h_0,\hh],
\end{align*}
which implies $w(z_0,\tau_0,h)>0$ for all $h \in (h_0,\hh]$, or equivalently, $z_0 \geq z^*(\tau_0,h_0)$. This  implies
$$\Big(\p_\tau \uh-\TT \uh\Big)(z_0,\tau_0,h_0) = \wU(e^{z_0},h_0).$$
Combining \eqref{uu_pb}, \eqref{Lphi} and \eqref{wwUwU} we get
$$
\Big(\p_\tau (\uh-\uu-\ep\zeta)-\TT (\uh-\uu-\ep\zeta)\Big)(z_0,\tau_0,h_0) \leq \wU(e^{z_0},h_0) - \wwU(e^{z_0})\leq 0,
$$
which contradicts \eqref{Luu>0}. So \eqref{v<=w} holds.
Sending $\ep\to0$ and then $\underline{h}\to0$ in \eqref{v<=w} we prove \eqref{uu}.
\end{proof}

By applying a similar argument, we can establish the following lemma:
\begin{lemma}\label{lem:uu}
For $\hh < \hh'$, we have
\begin{align*}
\uh(z,\tau,h) \leq u^{\hh'}(z,\tau,h), \quad (z,\tau,h)\in \R \times [0,T] \times (0,\hh].
\end{align*}
\end{lemma}

Since $\uh(z,\tau,h)$ is increasing in $\hh$ and uniformly bounded above for each $(z,\tau,h)\in \R \times (0,T] \times \R^+$, we define the limit function
\begin{align}\label{u_def}
u(z,\tau,h) := \lim_{\hh \to +\infty} \uh(z,\tau,h), \quad (z,\tau,h) \in \R \times [0,T] \times \R^+.
\end{align}

\begin{theorem}\label{u}
The function $u$ defined in \eqref{u_def} lies in $C^{2,1,1}(\R \times [0,T] \times \R^+)$ and solves Problem \eqref{u_pb}.
\end{theorem}

\begin{proof}
Letting $\hh \to \infty$ in \eqref{Luh}, we obtain
\begin{align}\label{TuR}
(u_\tau - \TT u)(z,\tau,h) = \wU(e^z, h^*(z,\tau) \vee h).
\end{align}
For any fixed $h_0 \in \R^+$, Schauder interior estimates (see \cite{Li96}, Theorem 4.9) imply $u(\cdot,\cdot,h_0) \in C^{2,1}(\R \times [0,T])$.

Furthermore, noting $u_h = -w$, we compute
\begin{align*}
\p_\tau u(z,\tau,h) &= \p_\tau u(z,\tau,h_0) - \int_{h_0}^h w_\tau(z,\tau,s) \, \d s, \\
\p_z u(z,\tau,h) &= \p_z u(z,\tau,h_0) - \int_{h_0}^h w_z(z,\tau,s) \, \d s.
\end{align*}
Since $w_z(\cdot,\cdot,h)$ and $w_\tau(\cdot,\cdot,h)$ are continuous in $\R \times (0,T]$ and uniformly bounded with respect to $h$, the Lebesgue dominated convergence theorem yields that $\p_\tau u$ and $\p_z u$ are continuous on $\R \times [0,T] \times \R^+$.
Moreover, combining \eqref{TuR}, we conclude that $u_{zz}$ is also continuous, hence $u \in C^{2,1,1}(\R \times [0,T] \times \R^+)$.

Finally, by the same arguments as in the proof of Lemma~\ref{lem:TuhR}, we verify that $u$ solves Problem \eqref{u_pb} using \eqref{TuR}.
\end{proof}

Define
\begin{align}\label{v_def}
v(y,\tau,h) := u(\ln y,\tau,h), \quad (y,\tau,h) \in \R^+ \times (0,T] \times \R^+.
\end{align}
Thanks to Lemma~\ref{u}, we obtain:
\begin{theorem}\label{v}
The function $v$ defined in \eqref{v_def} belongs to $C^{2,1,1}(\R^+ \times [0,T] \times \R^+)$ and is the solution to Problem \eqref{v_pb}.
\end{theorem}

To construct the solution to the original problem \eqref{V_pb} via inverse dual transformation, we now analyze several properties of the dual solution, including convexity and the asymptotic behavior of its derivative with respect to $y$ as $y \to 0$ and $y \to \infty$.

Observe that
\[
v_{yy}(y,\tau,h) = e^{-2z} \left( u_{zz}(z,\tau,h) - u_z(z,\tau,h) \right).
\]
To establish the convexity of $v$ in $y$, we first show:
\begin{lemma}\label{lem:unzz}
The solution $\uh$ to Problem~\eqref{uh_pb} defined by \eqref{uw} satisfies
\begin{align}\label{uzzuz}
\p_{zz} \uh(z,\tau,h) - \p_z \uh(z,\tau,h) \geq 0.
\end{align}
\end{lemma}

\begin{proof}
Set
\[
\psi^\hh(z,\tau,h) := \p_{zz} \uh(z,\tau,h) - \p_z \uh(z,\tau,h).
\]
We begin by verifying \eqref{uzzuz} for $h = \hh$. That is,
\begin{align}\label{psi>=0}
\psi^\hh(z,\tau,\hh) \geq 0.
\end{align}
Using the $C^{2+\alpha,1+\alpha/2}$ interior regularity for \eqref{fh_pb}, we see that $\psi^\hh(\cdot,\cdot,\hh)$ satisfies the exponential growth condition. Moreover, by direct computation:
\begin{align*}
\left\{
\begin{array}{ll}
\p_\tau \psi^\hh - \TT \psi^\hh = \wU_{yy}(e^z,\hh) e^{2z} \geq 0, & (z,\tau) \in \R \times (0,T], \\
\psi^\hh(z,0,\hh) = \wwU_T''(e^z) e^{2z} \geq 0, & z \in \R.
\end{array}
\right.
\end{align*}
Then the parabolic maximum principle yields \eqref{psi>=0}.

Next, we show that \eqref{uzzuz} holds on the free boundary:
\begin{align}\label{psiz}
\psi^\hh(z^*(\tau,h),\tau,h) \geq 0, \quad \text{for all } (\tau,h) \in [0,T] \times (0,h].
\end{align}
For any small $\ep>0$, define
\[
h_\varepsilon := \sup\left\{ h \in (0,\hh] \mid \exists \tau \in (0,T] \text{ such that } \psi^\hh(z^*(\tau,h),\tau,h) < -\varepsilon \right\}.
\]
We aim to prove $h_\varepsilon = 0$. Suppose for contradiction that $h_\varepsilon > 0$, and that for some $\tau_\varepsilon \in (0,T]$,
\begin{align}\label{psi>=-ep}
\psi^\hh(z^*(\tau_\varepsilon,h_\varepsilon),\tau_\varepsilon,h_\varepsilon) \leq -\varepsilon.
\end{align}
Let $z_\varepsilon := z^*(\tau_\varepsilon,h_\varepsilon)$.

Firstly, \eqref{psi>=0} implies $h_\ep< \hh$.
Note that
\begin{align*}
\left\{\begin{array}{ll}
(\p_\tau \psi^\hh  - \TT \psi^\hh)(z,\tau,h)  = \wU_{yy} (e^z,h) e^{2z}\geq 0,\quad z>z^*(\tau,h),\;h\in[h_\ep,\hh],\\
\psi^\hh (z^*(\tau,h),\tau,h)\geq -\ep,\quad h\in[h_\ep,\hh],\\
\psi^\hh (z,0,h)= \wwU_T'' (e^z) e^{2z}\geq 0,\quad z>z^*(0+,h)=\ln(U'(h)),\;h\in[h_\ep,\hh].
\end{array}\right.
\end{align*}
Applying the comparison principle we have
$$
\psi^\hh (z,\tau,h)\geq -\ep,\quad z\geq z^*(\tau,h),\;\tau\in [0,T],\;h\in[h_\ep,\hh].
$$

Since $w_\tau - \TT w + f(z,h) = 0$ for $z > z^*(\tau_\varepsilon,h_\varepsilon)$, and
\[
w(z_\varepsilon,\tau_\varepsilon,h_\varepsilon) = w_z(z_\varepsilon,\tau_\varepsilon,h_\varepsilon) = w_\tau(z_\varepsilon,\tau_\varepsilon,h_\varepsilon) = 0,
\]
and $z^*(\tau_\varepsilon,h_\varepsilon) < \ln U'(h_\varepsilon)$, it follows that
\[
\frac{\ka^2}{2} w_{zz}(z_\varepsilon^+,\tau_\varepsilon,h_\varepsilon) = f(z_\varepsilon,h_\varepsilon) \geq U'(h_\varepsilon) - e^{z_\varepsilon} > 0.
\]
Hence,
\begin{align}\label{psih}
\p_h \psi^\hh(z_\varepsilon^+,\tau_\varepsilon,h_\varepsilon) = - w_{zz}(z_\varepsilon^+,\tau_\varepsilon,h_\varepsilon) + w_z(z_\varepsilon,\tau_\varepsilon,h_\varepsilon) < 0.
\end{align}
This implies that there exists $h' > h_\varepsilon$ such that
\[
\psi^\hh(z_\varepsilon,\tau_\varepsilon,h_\varepsilon) > \psi^\hh(z_\varepsilon,\tau_\varepsilon,h') \geq -\varepsilon,
\]
which contradicts \eqref{psi>=-ep}. Therefore, $h_\varepsilon = 0$ and \eqref{psiz} holds.
\end{proof}

Now we prove that \eqref{uzzuz} holds for the case $z > z^*(\tau,h)$. Thanks to \eqref{psiz}, we have
\begin{align*}
\left\{
\begin{array}{ll}
(\p_\tau \psi^\hh - \TT \psi^\hh)(z,\tau,h) = \wU_{yy}(e^z,h) e^{2z} \geq 0, & z > z^*(\tau,h),\; h \in (0,\hh], \\
\psi^\hh(z^*(\tau,h),\tau,h) \geq 0, & h \in (0,\hh], \\
\psi^\hh(z,0,h) = \wwU_T''(e^z) e^{2z} \geq 0, & z > z^*(0+,h) = \ln(U'(h)),\; h \in (0,\hh].
\end{array}
\right.
\end{align*}
By the parabolic maximum principle, we conclude that
\[
\psi^\hh(z,\tau,h) \geq 0, \quad z \geq z^*(\tau,h),\; \tau \in [0,T],\; h \in [h_\varepsilon,\hh].
\]

Finally, we verify \eqref{uzzuz} for $z < z^*(\tau,h)$. Fix any $(z,\tau,h) \in \R \times (0,T] \times (0,\hh]$ such that $h < h^*(z,\tau)$. Since the function $\uh(\cdot,\cdot,h) - \uh(\cdot,\cdot,h^*(z,\tau))$ achieves its minimum at $(z,\tau)$, we have
\[
\uh(z,\tau,h) = \uh(z,\tau,h^*(z,\tau)), \quad \uh_z(z,\tau,h) = \uh_z(z,\tau,h^*(z,\tau)),
\]
and
\[
\uh_{zz}(z,\tau,h) \geq \uh_{zz}(z,\tau,h^*(z,\tau)).
\]
Hence,
\begin{align}\label{psipsi}
\psi^\hh(z,\tau,h) \geq \psi^\hh(z,\tau,h^*(z,\tau)) \geq 0.
\end{align}

\begin{proposition}\label{pro:vyy}
For any $(y,\tau,h) \in \R^+ \times (0,T] \times (0,h]$, we have
\begin{align}\label{vyy}
v_{yy}(y,\tau,h) > 0.
\end{align}
\end{proposition}

\begin{proof}
It suffices to show that for any $(z,\tau,h) \in \R \times (0,T] \times (0,h]$,
\begin{align}\label{uzzuz>0}
u_{zz}(z,\tau,h) - u_z(z,\tau,h) > 0.
\end{align}
Letting $\hh \to \infty$ in \eqref{uzzuz}, we obtain
\[
u_{zz}(z,\tau,h) - u_z(z,\tau,h) \geq 0.
\]
We now prove the inequality is strict. Define
\[
\psi(z,\tau,h) := u_{zz}(z,\tau,h) - u_z(z,\tau,h).
\]
From the regularity of $u$ and the PDE satisfied by $u$, we know that
\begin{align*}
\left\{
\begin{array}{ll}
(\p_\tau \psi - \TT \psi)(z,\tau,h) = \wU_{yy}(e^z,h) e^{2z} \geq 0, & z > z^*(\tau,h),\; \tau \in (0,T], \\
\psi(z^*(\tau,h),\tau,h) \geq 0, & \tau \in (0,T], \\
\psi(z,0,h) = \wwU_T''(e^z) e^{2z} \geq 0, & z > z^*(0+,h) = \ln(U'(h)).
\end{array}
\right.
\end{align*}
By the strong maximum principle, we obtain
\begin{align}\label{psi>0}
\psi(z,\tau,h) > 0, \quad z > z^*(\tau,h),\; \tau \in (0,T].
\end{align}

Moreover, analogous to the argument used for \eqref{psih}, we have
\[
\p_h \psi(z^*(\tau,h)^+,\tau,h) = -w_{zz}(z^*(\tau,h)^+,\tau,h) < 0, \quad \tau \in (0,T].
\]
Therefore, there exists $h' > h$ such that
\[
\psi(z^*(\tau,h),\tau,h) > \psi(z^*(\tau,h),\tau,h') \geq 0,
\]
which, together with \eqref{psi>0}, implies
\[
\psi(z^*(\tau,h),\tau,h) > 0, \quad \tau \in (0,T].
\]

Finally, for any $(z,\tau,h) \in \R \times [0,T] \times \R^+$ with $z < z^*(\tau,h)$, the function $u(\cdot,\cdot,h) - u(\cdot,\cdot,h^*(z,\tau))$ attains its minimum at $(z,\tau)$. Thus,
\[
u_z(z,\tau,h) = u_z(z,\tau,h^*(z,\tau)), \quad
u_{zz}(z,\tau,h) \geq u_{zz}(z,\tau,h^*(z,\tau)),
\]
which implies
\[
\psi(z,\tau,h) \geq \psi(z,\tau,h^*(z,\tau)) > 0.
\]
\end{proof}

\begin{proposition}\label{lem:vy}
For each $(\tau,h)\in (0,T]\times \R^+$, we have
\begin{align}\label{vy1}
\lim\limits_{y\to+\infty}v_y(y,\tau,h) = -\frac{bh}{r}(1-e^{-r\tau}),\\
\label{vy2}
\lim\limits_{y\to 0+}v_y(y,\tau,h) = -\infty.
\end{align}
\end{proposition}

\begin{proof}
Observe that $v_y(y,\tau,h)=e^{-z} u_z(z,\tau,h)$, where $z=\ln y$. Define
\[
\phi(z,\tau,h):= e^{-z} u_z(z,\tau,h) + \frac{bh}{r}(1-e^{-r\tau}).
\]
It suffices to prove:
\begin{align}
\label{p1}
\lim\limits_{z\to +\infty}\phi(z,\tau,h) = 0,\\
\label{p2}
\lim\limits_{z\to -\infty}\phi(z,\tau,h) = -\infty.
\end{align}

We first establish
\begin{align}\label{psi<=0}
\phi \leq 0.
\end{align}
From \eqref{TuR}, $\phi(z,\tau,h)$ satisfies the following PDE:
\begin{align}\label{psi_pb}
\left\{
\begin{array}{ll}
\p_\tau \phi - \JJ \phi = g(z,\tau,h), & (z,\tau) \in \R\times(0,T], \\
\phi(z,0,h) = g_T(z), & z \in \R,
\end{array}
\right.
\end{align}
where
\begin{align*}
\JJ \phi &:= \frac{1}{2} \kappa^2 \phi_{zz} + (\rho - r + \frac{1}{2} \kappa^2) \phi_z - r \phi, \\
g(z,\tau,h) &:= \wU_y(e^z,h^*(z,\tau)\vee h) + e^{-z} \wU_h(e^z,h^*(z,\tau)\vee h) \p_z(h^*(z,\tau)\vee h) + bh, \\
g_T(z) &:= \wwU_T'(e^z).
\end{align*}

We note that $\wU_h(e^z,h)> 0$ when $z < \ln U'(h)$, and
\[
\p_z(h^*(z,\tau)\vee h) =
\begin{cases}
\p_z h^*(z,\tau) \leq 0, & z < z^*(\tau,h), \\
0, & z > z^*(\tau,h),
\end{cases}
\]
Since $z^*(\tau,h) < \ln U'(h)$, we conclude:
\begin{align}\label{g2}
\wU_h(e^z,h^*(z,\tau)\vee h)\p_z(h^*(z,\tau)\vee h) \leq 0.
\end{align}
Using the bounds $-h \leq \wU_y \leq -bh$, we have $g(z,\tau,h) \leq 0$. Moreover, since $g_T(z) = \wwU_T'(e^z) \leq 0$, the parabolic maximum principle yields \eqref{psi<=0}.

To prove \eqref{p1}, fix $\epsilon > 0$. Note that for $z > \ln(U'(bh))$, we have $\wU_y(e^z,h) = -bh$ and $g(z,\tau,h) = 0$. Also, $g_T(z) \to 0$ as $z \to +\infty$. Choose $z_0 > \ln(U'(bh))$ such that $g_T(z) > -\epsilon/2$ for $z > z_0$. Then $\phi$ solves:
\begin{align}\label{psi_pb+}
\left\{
\begin{array}{ll}
\p_\tau \phi - \JJ \phi = 0, & (z,\tau) \in (z_0,+\infty) \times (0,T], \\
\phi(z,0,h) > -\epsilon/2, & z > z_0.
\end{array}
\right.
\end{align}

Define $M := \max_{\tau \in [0,T]} |\phi(z_0,\tau,h)|$. Then $-\epsilon/2 - M e^{z_0 - z}$ is a subsolution. Choosing $N = z_0 + \ln(2M/\epsilon)$ ensures
\[
\phi(z,\tau,h) \geq -\frac{\epsilon}{2} - M e^{z_0 - z} > -\epsilon, \quad z > N.
\]
Hence, \eqref{p1} holds.

To prove \eqref{p2}, fix $M > 0$ and define $z_0 := z^*(T,2Mr(1-e^{-r\tau})^{-1} + bh)$. Since $h^*(z,\tau)$ is strictly decreasing in $z$ and $\tau$, for $z \leq z_0$ and $s \in [0,T]$,
\[
h^*(z,s) \geq h^*(z_0,T) = 2Mr(1-e^{-r\tau})^{-1} + bh.
\]
Assuming $2Mr(1-e^{-r\tau})^{-1} + bh > h$, then $z_0 < z^*(s,h) < \ln U'(h)$ and for $z < z_0$,
\[
\wU_y(e^z,h^*(z,s)) = -h^*(z,s), \quad g(z,s,h) \leq -h^*(z,s) + bh \leq -2Mr(1-e^{-r\tau})^{-1}.
\]

Define $\lambda := \kappa^2 + \rho$, and the domain
\[
\Omega := \{(z,s) \mid z < z_0 - \lambda s,\; s \in (0,T]\} \subset (-\infty,z_0]\times[0,T].
\]
Then $\phi$ satisfies:
\begin{align}\label{psi_pb-}
\left\{
\begin{array}{ll}
\p_s \phi - \JJ \phi = g(z,s,h) \leq -2Mr(1-e^{-r\tau})^{-1}, & (z,s) \in \Omega, \\
\phi(z,s,h) < 0, & (z,s) \in \partial_p \Omega,
\end{array}
\right.
\end{align}
where $\partial_p \Omega$ denotes the parabolic boundary.

Define a supersolution:
\[
\Phi(z,s) := -2M(1-e^{-r\tau})^{-1}(1 - e^{z - z_0 + \lambda s})(1 - e^{-rs}).
\]
We compute:
\begin{align*}
\p_s \Phi - \JJ \Phi &= -2Mr(1-e^{-r\tau})^{-1}(1 - e^{z - z_0 + \lambda s}) \\
&\quad + 2Mr(1-e^{-r\tau})^{-1}(1 - e^{-rs}) e^{z - z_0 + \lambda s}(\lambda - \kappa^2 - \rho + r) \\
&\geq -2Mr(1-e^{-r\tau})^{-1}, \quad (z,s) \in \Omega.
\end{align*}
Since $\Phi = 0$ on $\partial_p \Omega$, we have $\phi \leq \Phi$ in $\Omega$.

Let $z_{\tau,h} := z_0 - \lambda \tau - \ln 2$, then
\[
\phi(z,\tau,h) \leq \Phi(z,\tau) = -2M(1 - e^{z - z_0 + \lambda \tau}) < -M, \quad z < z_{\tau,h}.
\]
This proves \eqref{p2}.
\end{proof}

\section{The Solution and the Free Boundary to the Original Problem}\label{sec:V}
\setcounter{equation}{0}

In this section, we present the proofs of \thmref{thm:V} and \thmref{thm:h}. Our strategy is to construct the solution and characterize the free boundary of the original problem \eqref{V_pb} using the solution and free boundary of the dual problem \eqref{v_pb}.

\subsection{Proof of Theorem~\ref{thm:V}}

Let $t = T - \tau$, and define
\begin{align}\label{V_def}
V(x,t,h) := \inf\limits_{y>0} (v(y,\tau,h) + yx),\quad \forall\; x > \frac{bh}{r}(1 - e^{-r(T - t)}),\; t \in [0,T).
\end{align}
We aim to prove that the function $V(x,t,h)$ defined in \eqref{V_def} solves Problem \eqref{V_pb}.

Using \eqref{vyy}, \eqref{vy1}, and \eqref{vy2}, the unique minimizer of \eqref{V_def}, denoted $I(x,t,h)$, satisfies
\begin{align*}
I(x,t,h) := \argmin\limits_{y>0} (v(y,\tau,h) + yx) = (v_y(\cdot,\tau,h))^{-1}(-x),\quad (x,t,h) \in \RR_T,
\end{align*}
where $(v_y(\cdot,\tau,h))^{-1}(\cdot)$ denotes the inverse function of $v_y(\cdot,\tau,h)$. Hence,
\begin{align*}
V(x,t,h) = v(I(x,t,h),\tau,h) + xI(x,t,h).
\end{align*}
This leads to the following derivatives:
\begin{align}\label{Vx}
V_x(x,t,h) &= v_y(I(x,t,h),\tau,h) I_x(x,t,h) + xI_x(x,t,h) + I(x,t,h) = I(x,t,h) > 0, \\
\label{Vxx}
V_{xx}(x,t,h) &= I_x(x,t,h) = \partial_x[(v_y(\cdot,\tau,h))^{-1}(-x)] = \frac{-1}{v_{yy}(I(x,t,h),\tau,h)} < 0, \\
\label{Vt}
V_t(x,t,h) &= -v_\tau(I(x,t,h),\tau,h), \\
V_h(x,t,h) &= v_h(I(x,t,h),\tau,h).
\end{align}
Thus, \eqref{V1} holds and $V \in C^{2,1,1}(\RR_T)$.

Moreover,
\begin{align*}
& \Big(-V_t + \frac{1}{2}\kappa^2 \frac{V_x^2}{V_{xx}} - \wU(V_x,h) - r x V_x + \rho V\Big)(x,t,h) \\
=& \Big(v_\tau - \frac{1}{2} \kappa^2 y^2 v_{yy} - (\rho - r) y v_y + \rho v - \wU(y,h)\Big)(I(x,t,h), \tau, h),
\end{align*}
showing that $V$ satisfies the variational inequality \eqref{V_pb1}.

We now verify the terminal condition:
\[
\lim\limits_{t \to T^-} V(x,t,h) = U_T(x)
\]
for any $x > 0$ and $h > 0$. Suppose by contradiction that this limit does not hold. Then there exist sequences $\tau_n \to 0^+$, $x_n \to x$, and $h_n \to h$ such that
\begin{align}\label{not}
K := \lim\limits_{n \to \infty} V(x_n, T - \tau_n, h_n) \neq U_T(x).
\end{align}
Let $y_n := (v_y(\cdot, \tau_n, h_n))^{-1}(-x_n)$. We analyze three cases:

\textbf{Case 1:} $y_n \to +\infty$ (or a sub sequence of $y_n$ tends to $+\infty$). For any fixed $y > 0$, for sufficiently large $n$ we have $y < y_n$ and thus $v_y(y, \tau_n, h_n) < v_y(y_n, \tau_n, h_n) = -x_n$. Taking $n \to \infty$ gives $\wwU_T'(y) = v_y(y, 0, h) \leq -x$, and further letting $y \to \infty$ yields $\wwU_T'(\infty) \leq -x$, contradicting $\wwU_T'(\infty) = 0$.

\textbf{Case 2:} $y_n \to 0+$ (or a sub sequence of $y_n$ tends to $0+$). Then for any $y > 0$, sufficiently large $n$ yields $y > y_n$, implying $v_y(y, \tau_n, h_n) > -x_n$, and thus $\wwU_T'(y) = v_y(y,0,h) \geq -x$. Taking $y \to 0+$ implies $\wwU_T'(0+) \geq -x$. If $U_T \not\equiv 0$, then $\wwU_T'(0+) = -\infty$, a contradiction. If $U_T \equiv 0$, then $\wwU_T(y) \equiv 0$, and since $v_y < 0$:
\[
V(x_n, T - \tau_n, h_n) = v(y_n, \tau_n, h_n) + y_n x_n > v(y, \tau_n, h_n) + y_n x_n \to 0,
\]
while
\[
V(x_n, T - \tau_n, h_n) \leq v(y, \tau_n, h_n) + y x_n \to y x.
\]
Since $y > 0$ is arbitrary, we conclude $\lim_{n \to \infty} V(x_n, T - \tau_n, h_n) = 0 = U_T(x)$, contradicting \eqref{not}.

\textbf{Case 3:} $y_n$ remains in $[a,b]$ for some $a,b > 0$. Then a subsequence, still denoted by $y_n$ that $y_n \to y_0 > 0$, which implies
\[
V(x_{n}, T - \tau_{n}, h_{n}) = v(y_{n}, \tau_{n}, h_{n}) + y_{n} x_{n} \to \wwU_T(y_0) + y_0 x \geq U_T(x),
\]
and for any $y > 0$,
\[
V(x_{n}, T - \tau_{n}, h_{n}) \leq v(y, \tau_{n}, h_{n}) + y x_{n} \to \wwU_T(y) + y x.
\]
Taking infimum over $y > 0$ yields $\lim_{n \to \infty} V(x_n, T - \tau_n, h_n) = U_T(x)$, again contradicting \eqref{not}.

We now prove the bounds in \eqref{V}. The growth condition \eqref{U2} implies
\[
\wwU(y), \; \wwU_T(y) \leq K\left[1 + \frac{\gamma}{1 - \gamma} \left(\frac{y}{K}\right)^{-\frac{1 - \gamma}{\gamma}}\right].
\]
Then by the comparison principle to the linear problem \eqref{uu_pb} we can prove
$$
\uu(z,\tau) \leq e^{\la \tau} K \left[ 1 + \frac{\gamma}{1-\gamma} \left( \frac{y}{K} \right)^{-\frac{1-\gamma}{\gamma}} \right]
$$
for some large $\la>0$, which combining \eqref{uu} gives
$$
v(y,\tau,h)\leq e^{\la \tau} K \left[ 1 + \frac{\gamma}{1-\gamma} \left( \frac{y}{K} \right)^{-\frac{1-\gamma}{\gamma}} \right].
$$
Plugging into \eqref{V_def} yields the second inequality in \eqref{V}.

To prove the first inequality, define
\[
\uv(y, \tau) = \frac{U(bh)}{\rho}(1 - e^{-\rho \tau}) - \frac{bh}{r}(1 - e^{-r \tau})y + K^{1/\theta} \frac{\theta}{1 - \theta} y^{1 - 1/\theta} e^{\lambda \tau} - K e^{-\rho \tau},
\]
where
\[
\lambda = -\frac{\kappa^2}{2} \frac{\theta - 1}{\theta^2} + (\rho - r) \frac{\theta - 1}{\theta} - \rho.
\]
This solves
\begin{align}\label{uv_pb}
\left\{
\begin{array}{ll}
\uv_\tau - \LL \uv = U(bh) - bhy, & (y, \tau) \in \R^+ \times (0,T], \\
\uv(y,0) = K^{1/\theta} \frac{\theta}{1 - \theta} y^{1 - 1/\theta} - K, & y > 0.
\end{array}
\right.
\end{align}
Since $\wU(y,h) \geq U(bh) - bhy$ and $\wwU_T(y) \geq \uv(y,0)$ (by \eqref{U3}), and $v$ is a supersolution of \eqref{uv_pb}, it follows that
\begin{align*}
V(x,t,h) \geq \inf_{y>0} (\uv(y,\tau) + yx)
= \frac{K}{1 - \theta} e^{\lambda \theta \tau} \left(x - \frac{bh}{r}(1 - e^{-r\tau})\right)^{1 - \theta}
+ \frac{U(bh)}{\rho}(1 - e^{-\rho \tau}) - K e^{-\rho \tau},
\end{align*}
proving the first inequality in \eqref{V}.

\subsection{Proof of Theorem~\ref{thm:h}}

Let $t = T - \tau$. Define the free boundary of the original problem \eqref{V_pb} as
\begin{align}\label{x^*}
x^*(t,h) := -v_y(e^{z^*(\tau,h)}, \tau, h), \quad t \in [0, T),\; h \in \R^+.
\end{align}
Since $v \in C^{2,1,1}(\R^+ \times (0,T] \times \R^+)$ and $z^* \in C((0,T] \times \R^+)$, it follows that $x^* \in C([0,T) \times \R^+)$.

Noting that $v_y < -\frac{bh}{r}(1 - e^{-r \tau})$, we obtain $x^*(t,h) > \frac{bh}{r}(1 - e^{-r(T - t)})$. Furthermore, since $z^*(\tau,h)$ is strictly decreasing in $h$ and $v_{yh} \leq 0$, we deduce that $x^*(t,h)$ is strictly increasing in $h$.

Define
\begin{align}\label{xHL}
x^H(t,h) := -v_y(U'(h), \tau, h), \quad x^L(t,h) := -v_y(U'(bh), \tau, h), \quad t \in [0, T),\; h \in \R^+.
\end{align}
Then both $x^H(t,h)$ and $x^L(t,h)$ are continuous in $(t,h)$ and strictly decreasing in $h$. Taking the limit as $t \to T-$, we obtain:
\[
x^H(T-, h) = -\widehat{U}_T'(U'(h)), \quad x^L(T-, h) = -\widehat{U}_T'(U'(bh)),
\]
and
\[
x^*(T-, h) = -\widehat{U}_T'(U'(h)),
\]
which establishes the terminal behavior of the free boundary.

\bibstyle{plainnat}


\bibliography{references}
\end{document}